\documentclass{article}

\usepackage{authblk}
\usepackage{amsfonts,amsmath,amssymb,mathtools,amsthm}
\usepackage{listings}
\usepackage{graphicx}
\usepackage{booktabs}
\usepackage{url}

\usepackage[text={6in,9in},centering]{geometry}

\renewcommand{\binom}[2]{\genfrac{(}{)}{0pt}{}{#1}{#2}}
\newcommand{\stirling}[2]{\genfrac{[}{]}{0pt}{}{#1}{#2}}

\def\bbbone{{\mathchoice {\rm 1\mskip-4mu l} {\rm 1\mskip-4mu l}
{\rm 1\mskip-4.5mu l} {\rm 1\mskip-5mu l}}}

\def\bbbe{\mathbb{E}}
\def\bbbp{\mathbb{P}}
\def\bbbz{\mathbb{Z}}

\def\bbbn{\mathbb{N}}
\def\ignore#1{}

\def\tC{\tilde{C}}

\newcommand{\mcL}{\mathcal{L}}
\newcommand{\p}{\mathbb{P}}

\newcommand{\R}{\mathcal{R}}

\newtheorem{lemma}{Lemma}
\newtheorem{remark}{Remark}
\newtheorem{theorem}{Theorem}
\newtheorem{proposition}{Proposition}

\title{Random derangements and the Ewens Sampling Formula}

\author[1]{Poly H. da Silva, Arash Jamshidpey, Simon Tavar\'e}
\affil[1]{Department of Statistics\\Columbia University\\1255 Amsterdam Avenue\\New York, NY 10027, USA}

\begin{document}
\maketitle


\begin{abstract}
We study derangements of $\{1,2,\ldots,n\}$ under the Ewens distribution with parameter $\theta$. We give the moments and marginal distributions of the cycle counts, the number of cycles, and asymptotic distributions for large $n$. We develop a $\{0,1\}$-valued non-homogeneous Markov chain with the property that the counts of lengths of spacings between the 1s have the derangement distribution. This chain,  an analog of the so-called Feller Coupling, provides a  simple way to simulate derangements in time independent of $\theta$ for a given $n$ and linear in the size of the derangement.
\end{abstract}

\noindent {\bf Keywords: } Feller Coupling, simulation, Poisson approximation, Poisson-Dirichlet distribution, probabilistic combinatorics

\noindent {\bf MSC: } 60C05,60J10,65C05, 97K20,97K60,65C40

\maketitle

\section{Introduction} 
The Ewens Sampling Formula~\cite{wje72} arose in population genetics as the joint probability distribution of the number of selectively neutral alleles $C_j(n)$ represented $j$ times in a sample of $n$ genes, for $j = 1,2,\ldots,n$. For positive integers $c_1,c_2,\ldots, c_n$ satisfying $\sum_{j=1}^n j c_j = n$, we have
\begin{equation}\label{esflaw}
    \bbbp_\theta(C_1(n) = c_1,\ldots,C_n(n) = c_n) = \frac{n!}{\theta_{(n)}}\,\prod_{j=1}^n \left(\frac{\theta}{j}\right)^{c_j} \frac{1}{c_j!},
\end{equation}
for $\theta \in (0,\infty)$, $\theta_{(n)} := \theta(\theta+1)\cdots(\theta+n - 1) = \Gamma(n+\theta)/\Gamma(\theta), n \geq 1$ and $\theta_{(0)} = 1$.\footnote{We define $\theta_{(-k)} = 0$, for $k \in \bbbn$.} In its original formulation, $\theta$ is a parameter related to the  rate at which novel alleles appear. In what follows we denote the law in (\ref{esflaw}) by ESF($\theta$); to simplify the notation, we suppress the $\theta$ in $\bbbp_{\theta}$ in what follows, where there is no cause for confusion.

The ESF has been studied extensively, and it arises in many different settings in probability and statistics. \cite[Chapter 41]{wjest97} provides an overview, \cite{arratia_logarithmic_2003} describes numerous applications in combinatorics, and~\cite{hc16} provides many other examples. Of particular interest here is its appearance as the distribution of the cycle counts of a $\theta$-biased permutation. Let $\pi$ be a permutation of $\{1,2,\ldots,n\}$ decomposed as a product of cycles. If $\pi$ is chosen uniformly with probability $1/n!$, then Cauchy's formula establishes that the cycle counts $(C_1(n),\ldots,C_n(n))$
have the ESF(1) law~\cite{vlg44}, and if a permutation $\pi$ having $k$ cycles is chosen with probability proportional to $\theta^k$, then the cycle counts have the ESF($\theta$) law. In this case,
\begin{equation}\label{pilaw}
\bbbp (\pi) = \frac{\theta^k}{\theta_{(n)}},   
\end{equation}
if the permutation $\pi$ has $k$ cycles.
See~\cite[Chapters 1 and 2.5]{arratia_logarithmic_2003} for more detailed discussion and history.

\subsection{Derangements}
Students of probability often meet derangements in the context of (versions of) the so-called hat-check problem~\cite[Chapter IV]{wf68}: 
$n$ diners leave their hats at a restaurant before their meal and  hats are returned at random after the meal. What is the probability that no diner gets back their own hat? Label the diners 1,2,\ldots, $n$ and construct a permutation $\pi$ by setting $\pi_j$ to be the label of the diner whose hat was returned to $j$. The question asks us to evaluate the probability that $\pi$ has no singleton cycles, and inclusion-exclusion is typically used to show that the required probability is
\begin{equation}\label{hats}
\bbbp_1(C_1(n) = 0) = \frac{D_n}{n!} = \sum_{l=0}^n (-1)^l \frac{1}{l!},
\end{equation}
where $D_n$ is the $n$th derangement number, the number of $n$-permutations with no fixed points.
The cycles of a derangement describe groups of diners who share hats among themselves, with no diner getting his own. The cycle counts  $(\tilde{C}_2(n),\ldots,\tilde{C}_n(n))$  of such a derangement have a distribution determined by
\begin{equation}\label{derlaw}
\mcL(\tilde{C}_2(n),\ldots,\tilde{C}_n(n)) = \mcL(C_2(n),\ldots,C_n(n) \vert C_1(n) = 0)
\end{equation}
The distribution on the right of (\ref{derlaw}) is determined by ESF(1) for a random permutation, and by ESF($\theta)$ for the biased case, when
(\ref{hats}) is replaced by
\begin{equation}\label{thetahats}
    \lambda_n(\theta) := \bbbp(C_1(n) = 0) = \frac{n!}{\Gamma(n+\theta)} \sum_{j=0}^n (-1)^j \frac{\theta^j}{j!}\, \frac{\Gamma(n+\theta-j)}{(n-j)!},
\end{equation}
with $\lambda_0(\theta) = 1, \lambda_1(\theta) = 0.$

\section{Properties of derangements} 
In this section we collect some results for derangements obtainable directly from~(\ref{derlaw}).

\subsection{Factorial moments of the cycle counts}
The falling factorial moments are straightforward to compute. For $r_2,r_3,\ldots,r_b \geq 0$ with $2r_2 + \cdots + b r_b = m \leq n$,
\begin{eqnarray*}
\lefteqn{\lambda_n(\theta)\, \mathbb{E}(\tC_2^{[r_2]} \cdots \tC_b^{[r_b]}) = \sum\nolimits{'}\,
c_2^{[r_2]} \cdots c_b^{[r_b]} \frac{n!}{\theta_{(n)}} \prod_{j = 2}^{n} \left(\frac{\theta}{j}\right)^{c_j}\,\frac{1}{c_j!}}\\
& = &  \frac{n!}{\theta_{(n)}}\, \prod_{j=2}^b \left(\frac{\theta}{j}\right)^{r_j}\,\frac{\theta_{(n-m)}}{(n-m)!} \
\sum\nolimits{''}\,
\frac{(n-m)!}{\theta_{(n-m)}} \prod_{j=2}^b \left(\frac{\theta}{j}\right)^{c_j'} \frac{1}{c_j'!} 
\prod_{j=b+1}^n \left(\frac{\theta}{j}\right)^{c_j} \frac{1}{c_j!} \\
& = & \frac{n!}{\theta_{(n)}}\, \prod_{j=2}^b \left(\frac{\theta}{j}\right)^{r_j}\,\frac{\theta_{(n-m)}}{(n-m)!}\ \lambda_{n-m}(\theta),
\end{eqnarray*}
since the last sum is just the probability that a random permutation of $(n-m)$ objects is a derangement; the sum in $\sum\limits{'}$ is over $2c_2 + \cdots + n c_n = n, c_2 \geq  r_2, \ldots,c_b \geq r_b, c_{b+1},\ldots,c_n \geq 0$, and the sum $\sum\limits{''}$ is over $c_2',\ldots,c_b',c_{b+1},\ldots,c_n \geq 0$ satisfying $ 2c_2' + \cdots + b c_b' +  (b+1)c_{b+1} + \cdots + n c_n = n-m$.
Hence
\begin{equation}\label{moments}
\mathbb{E}(\tC_2^{[r_2]} \cdots \tC_b^{[r_b]}) =  \bbbone(m \leq n)\,\frac{n!}{\lambda_n(\theta) \theta_{(n)}}\,\frac{\lambda_{(n-m)}(\theta) \theta_{(n-m)}}{(n-m)!} \, \prod_{j=2}^b \left(\frac{\theta}{j}\right)^{r_j}.
\end{equation}
In particular, for $j = 2,\ldots,n$,
\begin{equation}\label{ecjn*}
\bbbe \tC_j(n) = \frac{n!}{\lambda_n(\theta) \theta_{(n)}}\,\frac{\lambda_{(n-j)}(\theta) \theta_{(n-j)}}{(n-j)!}\, \frac{\theta}{j}.
\end{equation}
Note that $\bbbp(\tC_{n-1}(n) =  0) = 1$, and indeed $\bbbe \tC_{n-1}(n) = 0$.

\subsection{Distribution of the cycle counts}
To compute the distribution of the cycle counts, suppose that $X$ is a discrete random variable taking values in $\{0,1,2,\ldots,n\}$, with distribution $p_l = \bbbp(X = l), 0 \leq l \leq n$. Define
$$
u_j = \bbbe X^{[j]} = \sum_{l=j}^n l_{[j]}\, p_l, \quad j=1, 2, \ldots, n,
$$
where $l_{[j]} = l(l-1)\cdots (l-j+1)$ and $u_0 = 1$. Inverting this relationship gives
$$
p_r = \frac{1}{r!} \sum_{l=0}^{n-r} (-1)^l \,\frac{1}{l!}\, u_{r+l} = \frac{1}{r!} \sum_{i=r}^{n} (-1)^{i-r} \,\frac{1}{(i-r)!}\, u_{i}.
$$

Using the result in (\ref{moments}), choose $j \in \{2,\ldots,n\}, i \leq \lfloor n/j\rfloor$ and set
$$
u_i = \bbbe \tC_j(n)^{[i]} = \frac{n!}{\lambda_n(\theta) \theta_{(n)}}\,\frac{\lambda_{(n-ji)}(\theta) \theta_{(n-ji)}}{(n-ji)!} \, \left(\frac{\theta}{j}\right)^{i}.
$$
Then for $0 \leq r \leq \lfloor n/j \rfloor$,
\begin{eqnarray}
\bbbp(\tC_j(n) = r) & = & \left(\frac{\theta}{j}\right)^{r}\frac{1}{r!} \frac{n!}{\lambda_n(\theta) \theta_{(n)}} \nonumber \\
& & \qquad \times \sum_{i = r}^{\lfloor n/j\rfloor} n (-1)^{i-r} \frac{1}{(i-r)!} \frac{\lambda_{(n-ji)}(\theta) \theta_{(n-ji)}}{(n-ji)!} \, \left(\frac{\theta}{j}\right)^{i-r}.\label{cjnlaw}
\end{eqnarray}
The special case $j = n, r = 1$ is used in Section~\ref{singlecycle}.

\begin{remark}
Many of these results are well known in the case of random derangements, for which $\theta = 1$. For example, 
$$
\bbbp_1(\tC_2(n) = 0) = \frac{n!}{D_n}\,\sum_{i = 0}^{\lfloor n/2\rfloor} (-1)^i \frac{1}{i!} \left(\frac{1}{2}\right)^i \frac{D_{(n-2i)}}{(n-2i)!}.
$$
The integers  
$$
a(n) = D_n \, \bbbp_1(\tC_2(n) = 0), \quad n = 1, 2,3,\ldots 
$$
give the number of derangements of $n$ objects that have all cycles of length at least 3; computing the first few values gives
$$
a(2) = 0, a(3) = 2, a(4) = 6, a(5) = 24, a(6) = 160, a(7) = 1140, a(8) = 8988, \ldots.
$$
It is readily checked that this is (the start of)  sequence A038205 in the Online Encyclopedia of Integer Sequences~\cite{oeis2020}, where other formulae are provided.
\end{remark}

\subsection{The number of cycles}
The distribution of the number of cycles, $\tilde{K}_n$, may be found from the fact that the number $D(n,k)$ of derangements of size $n$ having $k$ cycles is 
$$
D(n,k) = \sum_{l=0}^k (-1)^l \binom{n}{l} \stirling{n-l}{k-l},
$$
where $\stirling{n}{k}$ is the unsigned Stirling number of the first kind. It follows that
\begin{eqnarray}
 \bbbp(\tilde{K}_n = k) & = & \frac{1}{\lambda_n(\theta)}\,\sum_{
 \genfrac{}{}{0pt}{}{\pi: |\pi| = k}{\pi \textrm{ a derangement}}
 } \bbbp(\pi) \nonumber \\
 &  = & \frac{1}{\lambda_n(\theta)}\,\sum_{
 \genfrac{}{}{0pt}{}{\pi: |\pi| = k}{\pi \textrm{ a derangement}}
 }
 \frac{\theta^k}{\theta_{(n)}} \qquad ({\rm from\ (\ref{pilaw})}) \nonumber \\
 & = & \frac{\theta^k D(n,k)}{\lambda_n(\theta) \theta_{(n)}}, \ k=1,2,\ldots, \lfloor n/2 \rfloor. \label{knlaw}
\end{eqnarray}
Note that
\begin{eqnarray*}
 \lambda_n(\theta) = \frac{1}{\theta_{(n)}} \sum_{k = 1}^n \theta^k D(n,k) & = & \frac{1}{\theta_{(n)}} \sum_{l=0}^n (-1)^l \binom{n}{l} \theta^l \, \sum_{k=l}^n \theta^{k-l} \stirling{n-l}{k-l} \\
 & = & \frac{1}{\theta_{(n)}} \sum_{l=0}^n (-1)^l \binom{n}{l} \theta^l \theta_{(n-l)}\\
 & = & \frac{n!}{\theta_{(n)}} \sum_{l=0}^n (-1)^l \frac{\theta^l}{l!} \frac{\theta_{(n-l)}}{(n-l)!},
\end{eqnarray*}
providing a direct validation of (\ref{thetahats}).
$\tilde{K}_n$ has mean
\begin{equation}\label{knmean}
\bbbe \tilde{K}_n = \bbbe (\tC_2(n) + \cdots + \tC_n(n)) =  \frac{1}{\lambda_n(\theta)} \frac{n!}{\theta_{(n)}}\,\sum_{j=2}^n \lambda_{(n-j)}(\theta) \frac{\theta_{(n-j)}}{(n-j)!}\, \frac{\theta}{j}
\end{equation}

\subsection{Properties derived from the Conditioning Relation}
 ESF($\theta$) may be represented as the law of independent Poisson random variables $Z_1,Z_2,\ldots,Z_n$ with 
 \begin{equation}\label{zjmean}
 \bbbe Z_j = x^j \theta/j \mbox{ for any } x > 0,
 \end{equation} 
 conditioned on $T_n := Z_1 + 2 Z_2 + \cdots + nZ_n = n$. This is known as the Conditioning Relation, and is exploited in the context of combinatorial structures in~\cite{arratia_logarithmic_2003}. The same relationship holds for derangements too: defining $T_{1n} = 2 Z_2 + \cdots + nZ_n$, we have
\begin{equation}\label{condlaw}
\mathcal{L}(\tC_2(n),\ldots,\tC_n(n)) = \mathcal{L}( Z_2,\ldots,Z_n \mid T_{1n} = n).
\end{equation}
To see this, note that for $c_2 \geq 0, \ldots, c_n \geq 0$ satisfying  $2c_2 + \cdots + n c_n = n,$
\begin{eqnarray}
\lefteqn{\bbbp(C_1(n) = 0,C_2(n) = c_2,\ldots,C_n(n) = c_n)  =  } \nonumber\\
& = & \bbbp(Z_1 = 0,Z_2=c_2,\ldots,Z_n = c_n \mid T_n = n) \nonumber\\
& = & \bbbp(Z_1 = 0,Z_2=c_2,\ldots,Z_n = c_n, Z_1 + T_{1n} = n)\,/\,\bbbp(T_n = n) \nonumber\\
& = & \bbbp(Z_1 = 0) \bbbp(Z_2=c_2,\ldots,Z_n = c_n \mid T_{1n} = n) \bbbp(T_{1n} = n) \,/\, \bbbp(T_n = n) \label{top}
\end{eqnarray}
while 
\begin{eqnarray}
\bbbp(C_1(n) = 0) & = & \bbbp(Z_1 = 0 \mid T_n = n) \nonumber \\
& = & \bbbp(Z_1 = 0,  T_n = n) \,/\, \bbbp(T_n = n) \nonumber \\
& = & \bbbp(Z_1 = 0) \bbbp(T_{1n} = n) \,/\,\bbbp(T_n = n) \label{bottom}
\end{eqnarray}
Dividing (\ref{top}) by (\ref{bottom}) and using (\ref{derlaw}) establishes (\ref{condlaw}).

The relationship in (\ref{condlaw}) means that asymptotic results can be read off from the general theory in~\cite{arratia_logarithmic_2003}. For example, the $\tC_j(n)$ are asymptotically independent Poisson random variables with mean $\theta/j$, which follows from (\ref{moments}) as well. We note for later use the consequence that
\begin{equation}\label{prob0lim}
\lambda_n(\theta) = \bbbp(C_1(n) = 0) \to \bbbp_{\theta}(Z_1 = 0) = e^{-\theta}, \quad n \to \infty.
\end{equation}

The largest cycles, when scaled by $n$ have asymptotically the Poisson-Dirichlet law with parameter $\theta$. Total variation estimates for the Poisson result also follow from \cite{arratia_logarithmic_2003}, and methods akin to those in~\cite{at92b} may be used to derive central limit results, for example. We will not pursue this further here.

\section{The Feller Coupling}

The Feller Coupling was introduced in \cite{abt92} as a way to generate the cycles in a growing permutation one at a time, and it has proved very useful in the study of the asymptotics of properties of the ESF; \cite{at92b} illustrates some of these. 
To describe the Feller Coupling, define independent Bernoulli random variables $\xi_i$ satisfying
\begin{equation}\label{xidist}
\bbbp(\xi_i = 1) = \frac{\theta}{\theta+i-1}, \quad \bbbp(\xi_i = 0) = \frac{i-1}{\theta+i-1}, \quad i \geq 1.
\end{equation}
The cycle counts are determined by the spacings between the 1s in realizations of $\xi_i, i \geq 1$. If we define
\begin{equation}\label{feller}
    C_j(n) = \#j - \mbox{spacings in }  1\xi_2 \xi_3 \cdots \xi_n 1,
\end{equation}
then
\begin{itemize}
    \item[(i)] The law of $(C_1(n),\ldots,C_n(n))$ is ESF($\theta$); and
    \item[(ii)] $Z_j = C_j(\infty) = \#j$ - \mbox{spacings in } $1 \xi_2 \xi_3 \cdots $ are independent Poisson-distributed random variables with $\bbbe Z_j =j/\theta$.
\end{itemize}
Further details may be found in~\cite{abt92} and~\cite[Chapter 5]{arratia_logarithmic_2003}.

In the spirit of the Feller Coupling, we will construct a sequence of random variables $\eta_1 = 1, \eta_2, \eta_3, \ldots$ with the property that for any $n$ the law of the counts of spacings between the 1s in $1\eta_2 \eta_3 \cdots \eta_n 1$ is precisely that of (\ref{derlaw}). As might be anticipated,  $\eta_n, \eta_{n-1}, \ldots, \eta_2, \eta_1 =1$ is no longer a sequence of independent random variables, but rather a Markov chain. We will identify the structure of this chain, and provide some applications of its use.

\subsection{A useful Markov chain}

To identify the Markov chain, for $j\in \bbbn$ let 
\[
\Delta_j^*=\{(a_j,a_{j-1},\ldots,a_1)\in \{0,1\}^j: a_j=0 \ \ and \ \nexists \  i < j  \  s.t.\ a_i=a_{i-1}=1\},
\]
and 
$$
\Delta_j=\{(a_j,a_{j-1},\ldots,a_1)\in \Delta_j^*: a_1=1\}
$$
Note that $\Delta_1=\emptyset$. Let $\xi_i, i \geq 1$ be the Bernoulli random variables defined in (\ref{xidist}), and define
\begin{eqnarray*}
\lambda_j(\theta)& := & \p((\xi_j,\xi_{j-1},\ldots,\xi_1)\in \Delta_j)\\
& = & \sum\limits_{(r_j,r_{j-1},\ldots,r_1)\in\Delta_j}\p(\xi_j=r_j,\xi_{j-1}=r_{j-1},\ldots,\xi_1=r_1),
\end{eqnarray*}
which for $j>1$ is the probability that a random $j$-permutation with parameter $\theta$ constructed according to the Feller Coupling is a derangement; we have seen that $\lambda_j(\theta)$ is given by (\ref{thetahats}). 

Define  $\R_1=\{(1)\}$ and for $j \geq 2$,
$$
\R_j=\{0,1\}^{j-1}\times \{1\}=\{(a_1,\ldots,a_j)\in \{0,1\}^j: a_j=1\}.
$$
For $1\leq i\leq n$ and $r=(r_n,\ldots,r_1)\in \R_n$, let $N_i(r)$ be the number of $i$-spacings in the sequence $(1\,r)$ (i.e., the number of sub-patterns $1 \, 0^{i-1} \, 1$ in $(1\,r$)), and define  
\[
\rho(a_1,\ldots,a_n)=\{r=(r_n,r_{n-1},\ldots,r_1)\in\R_n: N_1(r)=a_1,\ldots,N_n(r)=a_n\}.
\]

We seek to construct a random sequence of $0$s and $1$s, $\eta = (\eta_n,\ldots,\eta_2,\eta_1=1)$ such that if
$$
\tilde{C}_i(n) := N_i(\eta), i=2,\ldots,n,
$$
then
\begin{equation}\label{wanted}
\p(\tilde{C}_2(n)=c_2,\ldots,\tilde{C}_n(n)=c_n)=\p(C_1(n)=0,\ldots,C_n(n)=c_n |  C_1(n)=0).
\end{equation} 
Simplifying the r.h.s. of (\ref{wanted}), we have
\begin{eqnarray*}
\p(\tilde{C}_j(n)=c_j, 2 \leq j \leq n) & = &\p(C_1(n)=0, C_j(n)=c_j,2 \leq j \leq n)/\p(C_1(n)=0) \nonumber \\
& = & \lambda_n^{-1}(\theta)
\sum_{
{\genfrac{}{}{0pt}{}{(r_n,r_{n-1},\ldots,r_1)}{\in  \rho(0,c_2,\ldots,c_n)}}
} \p(\xi_n=r_n,\ldots,\xi_1=r_1).
\end{eqnarray*}
Note that if $(r_n,r_{n-1},\ldots,r_1)\in \rho(0,c_2,\ldots,c_n)$, then $r_n=r_2=0$. This suggests defining $\eta_n,\eta_{n-1},\ldots,\eta_2,\eta_1=1$ with law 
\begin{eqnarray*}
\p(\eta_n=r_n,\ldots,\eta_1=r_1) &=& \p(\xi_n=r_n,\ldots,\xi_1=r_1 \mid (\xi_n,\ldots, \xi_1)\in \Delta_n)\\
&=&\left\lbrace 
\begin{array}{l}
\lambda_n^{-1}(\theta) \, \p(\xi_n=r_n,\ldots,\xi_1=r_1),  \mbox{ if }  (r_n,\ldots,r_1)\in \Delta_n \\
0,  \mbox{ otherwise.}
\end{array}
\right. 
\end{eqnarray*}

By construction, $(\tilde{C}_2(n),\ldots,\tilde{C}_n(n))$ has the law of $(C_1(n),\ldots, C_n(n))$ conditioned on $C_1(n)=0$. Since $\xi_j$ are independent random variables, given $\eta_i$, the vectors $({\eta}_n,{\eta}_{n-1},\ldots,{\eta}_{i+1})$ and $({\eta}_{i-1},\ldots,{\eta}_{1})$ are independent and hence  ${\eta}_{n},{\eta}_{n-1},\ldots,{\eta}_{1}$ is a Markov chain, starting from $\eta_{n+1} = 1$.

More explicitly, for $3\leq i\leq n-1$, $(r_n,r_{n-1},\ldots,r_{i+2}, x)\in \Delta_{n-i}^*$ and $y\in\{0,1\}$, let 
\begin{eqnarray*}
 \tau_{i+1}(x,y) & := & \p({\eta}_i=y \mid {\eta}_{n+1}=1,{\eta}_n=r_n,\ldots,{\eta}_{i+2}=r_{i+2}, {\eta}_{i+1}=x)\\
&=& \frac{\p({\eta}_{n+1}=1,{\eta}_n=r_n,\ldots,{\eta}_{i+2}=r_{i+2}, {\eta}_{i+1}=x, {\eta}_{i}=y)}{\p({\eta}_{n+1}=1,{\eta}_n=r_n,\ldots,{\eta}_{i+2}=r_{i+2}, {\eta}_{i+1}=x)}
\end{eqnarray*}

We compute this for $x,y\in\{0,1\}$. Starting with the case $x=y=0$, and for $3\leq i\leq n-1$, we write
$\tau_{i+1}(0,0) = A / B,$
where
\begin{eqnarray*}
A & = & 
\lambda_n^{-1}(\theta) \p(\xi_{n+1}=1,\xi_{n}=r_{n},\ldots,\xi_{i+2}=r_{i+2},\xi_{i+1}=0) 
 \p((\xi_{i},\ldots,\xi_{1})\in\Delta_i) \\
& = & \lambda_n^{-1}(\theta) \, \p(\xi_{n+1}=1,\xi_{n}=r_{n},\ldots,\xi_{i+2}=r_{i+2},\xi_{i+1}=0)\lambda_i(\theta)
\end{eqnarray*}
and 
\begin{eqnarray*}
B & = & \lambda_n^{-1}(\theta) \p(\xi_{n+1}=1,\xi_{n}=r_{n},\ldots,\xi_{i+2}=r_{i+2},\xi_{i+1}=0) \\
& & \quad \times \{\p((\xi_{i},\ldots,\xi_{1})\in\Delta_i)+\p(\xi_i=1)\p((\xi_{i-1},\ldots,\xi_{1})\in\Delta_{i-1})\} \\
& = & \lambda_n^{-1}(\theta)\,\p(\xi_{n+1}=1,\xi_{n}=r_{n},\ldots,\xi_{i+2}=r_{i+2},\xi_{i+1}=0) \\
&& \quad \times\left\{\lambda_i(\theta)+\frac{\theta}{\theta+i-1}\lambda_{i-1}(\theta)\right\}
\end{eqnarray*}
so that
$$
\tau_{i+1}(0,0) =  \frac{\lambda_i(\theta)}{\left\{\lambda_i(\theta)+\frac{\theta}{\theta+i-1}\lambda_{i-1}(\theta)\right\}}.
$$
On the other hand,
$$
\p({\eta}_i=0 \mid{\eta}_{i+1}=0)  =  \p({\eta}_{i+1}=0,{\eta}_{i}=0)/\p({\eta}_{i+1}=0) 
= C/D,
$$
where
\begin{eqnarray*}
C & = & \lambda_n^{-1}(\theta)\,\p((\xi_{n},\ldots,\xi_{i+2})\in\Delta_{n-i-1}^*,\xi_{i+1}=0) 
 \p((\xi_{i},\ldots,\xi_{1})\in\Delta_{i}) \\
& = & \lambda_n^{-1}(\theta)\, \p((\xi_{n},\ldots,\xi_{i+2})\in\Delta_{n-i-1}^*,\xi_{i+1}=0)\lambda_i(\theta)
\end{eqnarray*}
and
\begin{eqnarray*}
D & = & \lambda_n^{-1}(\theta)\p((\xi_{n},\ldots,\xi_{i+2})\in\Delta_{n-i-1}^*,\xi_{i+1}=0)\\ 
& & \quad\times \{\p((\xi_{i},\ldots,\xi_{1})\in\Delta_{i})+\p(\xi_i=1)\p((\xi_{i-1},\ldots,\xi_{1})\in\Delta_{i-1})\} \\
& = & \lambda_n^{-1}(\theta) \p((\xi_{n},\ldots,\xi_{i+2})\in\Delta_{n-i-1}^*,\xi_{i+1}=0)\, \left\{\lambda_i(\theta)+\frac{\theta}{\theta+i-1}\lambda_{i-1}(\theta)\right\}
\end{eqnarray*}
Hence
$$
\p({\eta}_i=0 \mid{\eta}_{i+1}=0) = \frac{\lambda_i(\theta)}
{\left\{\lambda_i(\theta)+\frac{\theta}{\theta+i-1}\lambda_{i-1}(\theta)\right\}}.
$$

Similarly we can deduce that
\[
\tau_i(0,1)=\frac{\frac{\theta\lambda_{i-1}(\theta)}{\theta+i-1}}{\lambda_i(\theta)+\frac{\theta\lambda_{i-1}(\theta)}{\theta+i-1}}=\p({\eta}_i=1 \mid {\eta}_{i+1}=0),
\]
while
\[
\tau_i(1,1) = 0 = \p({\eta}_i=1 \mid {\eta}_{i+1}=1)
\]
and
\[
\tau_i(1,0) = 1 = \p({\eta}_i=0 | {\eta}_{i+1}=1).
\]

We summarise the discussion as follows. 
\begin{theorem}
(i) For each $n \geq 3$, the sequence of random variables ${\eta}_{n+1}=1,{\eta}_{n},\ldots,{\eta}_{2},{\eta}_{1}=1$ is a non-homogeneous Markov chain with transition matrices
$$
P_{r} := \left(
\begin{array}{cc}
\bbbp(\eta_r = 0 \mid \eta_{r+1} = 0) & \bbbp(\eta_r = 1 \mid \eta_{r+1} = 0) \\
\bbbp(\eta_r = 0 \mid \eta_{r+1} = 1) & \bbbp(\eta_r = 1 \mid \eta_{r+1} = 1)
\end{array}
\right)
$$
given by
\begin{equation}\label{prdef}
P_{r} = \left(
\begin{array}{cc}
\displaystyle\frac{(\theta+r-1)\lambda_r(\theta)}{(\theta+r-1)\lambda_r(\theta) + \theta \lambda_{r-1}(\theta)} &  \displaystyle\frac{\theta\lambda_{r-1}(\theta)}{(\theta+r-1)\lambda_r(\theta) + \theta \lambda_{r-1}(\theta)} \\
&\\
1 & 0
\end{array}
\right),
\end{equation}
for $r = n-1,\ldots,3$,
$$
P_{2} = \left(
\begin{array}{cc}
1 & 0\\
1 & 0
\end{array}
\right),
\quad
P_{1} = \left(
\begin{array}{cc}
0 & 1\\
0 & 1
\end{array}
\right).
$$

(ii) The counts $\tilde{C}_j(n), j=2,\ldots,n$ of the j-spacings between consecutive 1s in the sequence $1 \eta_n \cdots \eta_2 1$ have joint distribution given by (\ref{derlaw}).
\end{theorem}

\subsection{The ordered cycles}\label{sectordered}
The $\eta$ process generates the length of cycles in an $n$-derangement in order, starting from the artificial boundary at $\eta_{n+1} = 1$. Denoting the length of the first cycle by $A_1(n)$, we have
\begin{eqnarray*}
\bbbp(A_1(n) > l) & = & \bbbp(\eta_{n+1} = 1, \eta_n = 0, \ldots, \eta_{n-l+1} = 0) \nonumber \\
& = & \prod_{r = n-l+1}^{n-1} \frac{(\theta+r -1) \lambda_r(\theta)}{(\theta+r -1) \lambda_r(\theta) + \theta \lambda_{r-1}(\theta)},
\end{eqnarray*}
which is readily computed. Some numerical illustrations appear in Table~\ref{Tab8}.

When $n$ is large, we have for $x \in (0,1)$
\begin{eqnarray*}
\log \bbbp(A_1(n) > \lfloor nx \rfloor) & = & - \sum_{r = n- \lfloor nx \rfloor +1}^{n-1} \log\left( 1 + \frac{\theta}{\theta+r-1}\,
\frac{\lambda_{r-1}(\theta)}{\lambda_r(\theta)}\right)\nonumber \\
& \sim & - \theta \sum_{r = n(1-x)}^{n-1} \frac{1}{\theta+r - 1} \frac{\lambda_{r-1}(\theta)}{\lambda_r(\theta)} \nonumber \\
& \sim & -\theta \int_{1-x}^1 u^{-1} du =  \theta \log(1-x), \label{a1nlim}
\end{eqnarray*}
using (\ref{prob0lim}). It follows that $n^{-1} A_1(n)$ has asymptotically a Beta distribution with density $\theta (1-x)^{\theta - 1}, 0 < x < 1.$ The joint law of the ordered spacings may be used in a similar way to show directly that $n^{-1}(A_1(n),A_2(n),\ldots)$ has asymptotically the GEM distribution with parameter $\theta$; see~\cite[Chapter 5.4]{arratia_logarithmic_2003}.

\section{Simulating derangements}
While we have a good understanding of the asymptotics of the distribution of cycle counts, for small values of $n$  simulation may be a useful approach to answer more detailed questions where explicit results are hard to find. Simulating derangements for the uniform case ($\theta = 1$) is a classical problem, and there have been many suggested methods, including~\cite{akl_new_1980}, \cite{martinez_generating_2008-1} and  \cite{Merlini2007} which use a modification of the Fisher-Yates algorithm for random permutations and a rejection step,  and improved by \cite{Mikawa2017}. \cite{mendonca_efficient_2020} exploits two different techniques, one based on random restricted transpositions and one on sequential importance sampling. We are not aware of explicit methods for the case of arbitrary $\theta$,
but the Markov chain approach provides an efficient way to do this.

\subsection{Rejection methods}

There are at least two such methods. For example, we can use the Feller Coupling to simulate $(C_1(n),C_2(n),\ldots,C_n(n))$ from ESF($\theta$) and  take $(\tilde{C}_2(n),\ldots,\tilde{C}_n(n)) = (C_2(n),\ldots,C_n(n))$ as an observation from (\ref{derlaw}) if $C_1(n) = 0$. The acceptance probability is $\lambda_n(\theta)$, which is $\approx e^{- \theta}$, so this strategy is slow if $\theta$ is large. Indicative results are shown in Table~\ref{Tab1}.

\begin{table}[htbp]
\caption{Rejection method. Derangements of size $n$, estimates based on 10,000 accepted runs.}
\centering
	\begin{tabular}{cccccccccc}
	\toprule
	\multicolumn{1}{c}{} & \multicolumn{3}{c}{$\theta = 0.5$} & \multicolumn{3}{c}{$\theta = 1.0$} &  \multicolumn{3}{c}{$\theta = 5.0$} \\
	 & time & accept & theory &  time & accept & theory &  time & accept & theory \\
	$n$ & (secs) & rate & (\ref{thetahats}) & (secs) & rate & (\ref{thetahats}) & (secs) & rate & (\ref{thetahats}) \\
	\cmidrule(l){2-4} \cmidrule(l){5-7} \cmidrule(l){8-10}
	10 & 0.38 & 0.590 & 0.591 & 0.62 & 0.372 & 0.368 & 9.58 & 0.024 & 0.023\\
	50 & 1.40 & 0.607 & 0.604 & 2.17 & 0.372 & 0.368 & 86.19 & 0.010 & 0.010\\
	250 & 6.89 & 0.600 & 0.606  & 10.71 & 0.367 & 0.368 & 532.1 &  0.007 & 0.007 \\
	\bottomrule
    \end{tabular}\label{Tab1}
\end{table}

The Conditioning Relation (\ref{condlaw}) provides another approach: the naive implementation takes $x = 1$ in (\ref{zjmean}), and simulates independent Poisson random variables $Z_2,\ldots,Z_n$ with $\bbbe Z_j = \theta/j$ and accepts $(Z_2,\ldots,Z_n)$ as an observation of the counts  $(\tC_2(n),\ldots,\tC_n(n))$ if $T_{1n} = 2 Z_2 + \cdots + n Z_n = n$. 
The acceptance probability is $\bbbp(T_{1n} = n)$; for large $n$, \cite[Theorem 4.13]{arratia_logarithmic_2003} shows that 
$ n \bbbp(T_{1n} = n) \sim e^{-\gamma \theta}/\Gamma(\theta)$,
where $\gamma$ is Euler's constant. We can do much better by adapting the argument in \cite[Section 5]{arratia_simulating_2018} by choosing $x = x(n)$ more carefully:  choose $c$ as the solution of the equation $\theta(1-e^{-c}) = c$, and set $ x = e^{-c/n}$. We then have
\begin{equation}\label{abtrate}
n \bbbp(T_{1n} = n) \sim e^{-\gamma \theta} e^{u(c)}/\Gamma(\theta),  \quad n \to \infty, 
\end{equation}
where $u(c) = -c + \theta \int_0^1 v^{-1} (1 - e^{-cv}) dv.$ The quantity $e^{u(c)}$ is the asymptotic factor by which the acceptance rate increases compared to the naive rate when $x = 1, c = 0$. When $\theta = 5$, this is 379.6, indicating a dramatic speed up over the naive version. Indicative results are shown in Table~\ref{Tab2}.

\begin{table}[htbp]
\caption{Conditioning Relation method. Derangements of size $n$, estimates based on 10,000 accepted runs. Values of $c$: -1.256 ($\theta = 0.5$), 0 ($\theta = 1$), 4.965 ($\theta  = 5$).}
\centering
\begin{tabular}{cccccccccc}
	\toprule
	\multicolumn{1}{c}{} & \multicolumn{3}{c}{$\theta = 0.5$} & \multicolumn{3}{c}{$\theta = 1.0$} &  \multicolumn{3}{c}{$\theta = 5.0$} \\
	 & time & accept & theory &  time & accept & theory &  time & accept & theory \\
	$n$ & (secs) & rate & (\ref{abtrate}) & (secs) & rate & (\ref{abtrate}) & (secs) & rate & (\ref{abtrate}) \\
	\cmidrule(l){2-4} \cmidrule(l){5-7} \cmidrule(l){8-10}
	10 & 2.46 & 0.059 & 0.061 & 2.59 & 0.055 & 0.056 & 4.41 & 0.032 & 0.088\\
	50 & 62.84 & 0.012 & 0.012 & 68.04 & 0.011 & 0.011 & 51.23 & 0.015 & 0.018\\
	250 & 1730 & 0.002 & 0.002 & 1884 & 0.002 & 0.002  & 1223 & 0.003 & 0.004\\
	\bottomrule
    \end{tabular}\label{Tab2}
\end{table}

Thus one of these methods is slow for large $n$, the other for large $\theta$. In contrast, the Markov chain approach provides a method that is acceptable for large $n$ and $\theta$. 

\subsection{Simulating derangements via the Markov chain}

It is straightforward to use the transition mechanism from Theorem 1 to generate a derangement from the spacings between the 1s in the Markovian sequence $1 \eta_n \eta_{n-1} \ldots \eta_2 1$. Indicative results are shown in Table~\ref{Tab3}. As anticipated, the run time of the Markov chain method is essentially constant as a function of $\theta$ for a fixed value of $n$, a property obviously not shared by the rejection methods. Comparing timings of these methods (which were implemented in {\sf R}) depends of course on the details of the code and the computer they are run on, so they should only be viewed as relative. It is interesting to note that the acceptance rate of the Conditioning Relation method is not monotone in $\theta$, because of the nature of the conditioning event. The first rejection method is sometimes faster than the Markov chain method, presumably because of the simpler coding required.

\begin{table}[htbp]
\caption{Markov chain method. Derangements of size $n$, estimates based on 10,000 accepted runs.}
\centering
\begin{tabular}{cccccccccc}
	\toprule
	\multicolumn{1}{c}{} & \multicolumn{3}{c}{Run time (secs)} \\
	\multicolumn{1}{c}{$n$} & \multicolumn{1}{c}{$\theta = 0.5$} & \multicolumn{1}{c}{$\theta = 1.0$} &  \multicolumn{1}{c}{$\theta = 5.0$} \\
		\cmidrule(l){2-4} \cmidrule(l){5-7} \cmidrule(l){8-10}
	10 &  0.686 & 0.654 & 0.642 \\
	50 & 4.665 & 4.636 & 4.752 \\
	250 & 32.98 & 32.91 & 33.57 \\
	\bottomrule
    \end{tabular}\label{Tab3}
\end{table}

\subsection{Examples}\label{examples}
We give four examples of the use of the simulation, for one of which the analytical answer is known, and for three of which it is not.\footnote{{\sf R} code for performing the computations described in the paper may be obtained from ST.}

\subsubsection{The probability of a single cycle}\label{singlecycle}

Since 
$$
\bbbp(\tC_n(n) = 1) = \bbbp(C_n(n) = 1 \mid C_1(n) = 0) = \bbbp(C_n(n) = 1) / \bbbp(C_1(n) = 0),
$$ 
we obtain
\begin{equation}\label{prob1cycle}
\bbbp(\tC_n(n) = 1) = \frac{n!}{\theta_{(n)}} \, \frac{\theta}{n} \, \frac{1}{\lambda_n(\theta)}, 
\end{equation}
which also follows from (\ref{ecjn*}) because $\bbbp(\tC_n(n) = 1) =  \bbbe \tC_n(n)$.

The asymptotics of (\ref{prob1cycle}) follow readily, using the fact that 
$n^{-\alpha} \Gamma(n+\alpha)/\Gamma(n) \to~1$ as  $n \to \infty$ to obtain
\begin{equation}\label{limprob1cycle}
\bbbp(\tC_n(n) = 1) \sim \Gamma(\theta+1) \left( \frac{e}{n}\right)^\theta, \quad n \to \infty.
\end{equation}
In Table~\ref{Tab8} some representative simulated and exact values are shown.

\begin{table}[htbp]
\caption{Probability that a derangement has a single cycle. Simulations use the Markov chain method, and estimates are based on 100,000 runs.}
\centering
\begin{tabular}{cccccccccc}
	\toprule
	\multicolumn{1}{c}{} & \multicolumn{3}{c}{$\theta = 0.5$} & \multicolumn{3}{c}{$\theta = 1.0$} &  \multicolumn{3}{c}{$\theta = 5.0$} \\
	 & sim & exact & asymp & sim  & exact  & asymp & sim  & exact & asymp \\
	$n$ &  & (\ref{prob1cycle})  & (\ref{limprob1cycle}) &  & (\ref{prob1cycle}) & (\ref{limprob1cycle}) &  &(\ref{prob1cycle}) & (\ref{limprob1cycle}) \\
	\cmidrule(l){2-4} \cmidrule(l){5-7} \cmidrule(l){8-10}
	10 & 0.476 & 0.480 & 0.462 & 0.270 & 0.272 & 0.272 & 0.021 & 0.021 & 0.178\\
	50 & 0.211 & 0.208 & 0.207 & 0.054 & 0.054 & 0.054 & 5$\times10^{-5}$ & 3.29$\times 10^{-5}$ & 5.70$\times 10^{-5}$\\
	250 & 0.092 & 0.093 & 0.092 & 0.011 & 0.011 & 0.011 & 0.0 & 1.62$\times 10^{-8}$ & 1.82$\times 10^{-8}$\\
	\bottomrule	
    \end{tabular}\label{Tab4}
\end{table}

\subsubsection{The probability that all cycle lengths are distinct}

This is a variant of the problem discussed in~\cite{arratia_simulating_2018}, for which there is no easy analytical answer.  The difference between the number of cycles and the number of distinct cycle lengths is 
$$
\tilde{D}_n = \sum_{j=2}^n (\tC_j(n) - 1)_+,
$$ 
where $(x)_+ = \max(0,x)$. We want $\bbbp(\tilde{D}_n = 0)$, which can be estimated by simulation. 

In \cite[eq. (10)]{arratia_simulating_2018} it is shown that for a permutation having the ESF($\theta$) distribution, the asymptotic probability that it has no repeated cycle lengths is $e^{-\gamma \theta}/\Gamma(\theta~+1).$ A modification of that argument shows that for derangements,
$$
\tilde{D}_n \Rightarrow \tilde{D} = \sum_{j \geq 2} (Z_j - 1)_+,
$$
where the $Z_j$ are the familiar independent Poisson random variables with $\bbbe Z_j = \theta/j$, so that
\begin{eqnarray}
\bbbp(\tilde{D}_n = 0) & \to & \bbbp(Z_j \leq 1, j \geq 2) 
 =  \prod_{j \geq 2} e^{-\theta/j}(1+\theta/j) \nonumber \\
& = & \frac{1}{e^{-\theta} (1+\theta)} \, \frac{e^{- \gamma \theta}}{\Gamma(\theta+1)} 
 =  \frac{e^{-\theta(\gamma - 1)}}{\Gamma(\theta+2)} \label{Dlim0}
\end{eqnarray}
Some representative values are given in Table~\ref{Tab5}.

 \begin{table}[htbp]
\caption{$\bbbp(\tilde{D}_n = 0)$, the probability that a derangement has distinct cycle lengths. Simulations using the Markov chain method, and estimates are based on 100,000 runs. The last row comes from (\ref{Dlim0}).}
\centering
\begin{tabular}{cccc}
	\toprule
	$n$ & $\theta = 0.5$ & $\theta = 1.0$ &  $\theta = 5.0$ \\
	\cmidrule(l){2-2} \cmidrule(l){3-3} \cmidrule(l){4-4} 
	10 & 0.885 & 0.774 & 0.357\\
	50 & 0.920 & 0.776 & 0.091\\
	250 & 0.927 & 0.765 & 0.028 \\
	$\infty$ & 0.929 & 0.763 & 0.012\\
	\bottomrule
    \end{tabular}\label{Tab5}
\end{table}

\subsubsection{The probability that all cycle lengths are even, or odd}
The third example is a little more subtle. We want to calculate
$$
\alpha_n = \bbbp\Big(\sum_{j \textrm{ even}} \tC_j(n) = 0\Big),\quad \beta_n = \bbbp\Big(\sum_{j \textrm{ odd}} \tC_j(n) = 0\Big),
$$
the probability that a derangement has all odd or all even cycle lengths, respectively. Clearly, if $n$ is odd then $\beta_n = 0$. We estimated $\alpha_n$ and $\beta_n$ by simulation, and some representative values are given in Table~\ref{Tab6}.

 \begin{table}[htbp]
\caption{The probability $\alpha_n$ that a derangement of length $n$ has all odd cycle lengths, and the probability $\beta_n$ that it has all even cycle lengths. Simulations using the Markov chain method, and estimates are based on 100,000 runs.}
\centering
\begin{tabular}{ccccccc}
	\toprule
	\multicolumn{1}{c}{} & \multicolumn{2}{c}{$\theta = 0.5$} & \multicolumn{2}{c}{$\theta = 1.0$} &  \multicolumn{2}{c}{$\theta = 5.0$} \\
	\cmidrule(l){2-3} \cmidrule(l){4-5} \cmidrule(l){6-7} 
	$n$ & $\alpha_n$ & $\beta_n$ & $\alpha_n$ & $\beta_n$ & $\alpha_n$ & $\beta_n$\\
	10 & 0.162 & 0.777 & 0.185 & 0.666 & 0.071 & 0.496 \\
	11 & 0.469 &  & 0.278 &  & 0.062 &  \\
	50 & 0.173 & 0.513 & 0.105 & 0.306 & 0.004 & 0.033\\
	51 & 0.261 &  & 0.114 &  & 0.004 & \\
	250 & 0.133 & 0.342 & 0.049 & 0.138 & 9$\times 10^{-5}$ & 8.9$\times 10^{-4}$\\
	251 & 0.160 &  & 0.051 &  & 6$\times 10^{-5}$ & \\
	\bottomrule
    \end{tabular} \label{Tab6}
\end{table}

\subsubsection{The ordered cycle lengths}\label{orderedex}

In Section~\ref{sectordered} we discussed briefly the size of the first cycle generated by the Markov chain. In Table~\ref{Tab8}, we compare  some properties of the longest cycle length for different values of $n$ and $\theta$, and in Table~\ref{Tab9} give some monotonicity properties of the ordered lengths. In the Appendix we provide proofs for some conjectures motivated by the simulation results.

 \begin{table}[htbp]
\caption{Probability $o_n$ that the largest cycle length is the first, the mean length $\bbbe A_1(n)$ of the first cycle, and the mean length $\bbbe L_1(n)$ of the longest cycle. Simulations use the Markov chain method, and estimates are based on 100,000 runs.}
\centering
\begin{tabular}{cccccccccc}
	\toprule
	\multicolumn{1}{c}{} & \multicolumn{3}{c}{$\theta = 0.5$} & \multicolumn{3}{c}{$\theta = 1.0$} &  \multicolumn{3}{c}{$\theta = 5.0$} \\
	$n$ & $o_n$ & $\bbbe A_1(n)$  & $\bbbe L_1(n)$ & $o_n$ & $\bbbe A_1(n)$  & $\bbbe L_1(n)$& $o_n$ & $\bbbe A_1(n)$   &  $\bbbe L_1(n)$ \\
	\cmidrule(l){2-4} \cmidrule(l){5-7} \cmidrule(l){8-10}
	10 & 0.847 & 7.64 & 8.16 & 0.766 & 6.45 & 7.17 & 0.604 & 3.91 & 4.79 \\
	50 &  0.775 & 34.33 & 38.43 & 0.652 & 26.51 & 32.15 &0.356  & 10.78 & 16.99 \\
	250 & 0.761 & 167.8  & 190.2  & 0.630 & 126.70 & 157.00 & 0.311 & 44.27 & 76.58 \\
	\bottomrule
\end{tabular}\label{Tab8}
\end{table}

 \begin{table}[htbp]
\caption{Probability that a derangement has weakly decreasing ordered cycle lengths ($\searrow$) or weakly increasing ordered cycle lengths ($\nearrow$). Simulations use the Markov chain method, and estimates are based on 100,000 runs.}
\centering
\begin{tabular}{ccccccc}
	\toprule
	\multicolumn{1}{c}{} & \multicolumn{2}{c}{$\theta = 0.5$} & \multicolumn{2}{c}{$\theta = 1.0$} &  \multicolumn{2}{c}{$\theta = 5.0$} \\
	$n$ & $\searrow$ & $\nearrow$  & $\searrow$ &$\nearrow$   & $\searrow$  & $\nearrow$  \\
	\cmidrule(l){2-3} \cmidrule(l){4-5} \cmidrule(l){6-7}
	10 & 0.833 & 0.646 & 0.730 & 0.475 & 0.486 & 0.216 \\
	50 & 0.666 & 0.287 & 0.433 & 0.101 & 0.023 & 3.5$\times 10^{-4}$\\
	250 & 0.561 & 0.130 & 0.257 & 0.020 & 4.2$\times 10^{-4}$ & 0.000\\
	\bottomrule
    \end{tabular} \label{Tab9}
\end{table}

\section{Discussion}
Our work was motivated by understanding how the Feller Coupling might be adapted to simulate derangements under the Ewens Sampling Formula with arbitrary parameter $\theta$. The result is a $\{0,1\}$-valued non-homogeneous Markov chain $\eta_{n+1} = 1, \eta_n,\eta_{n-1},\ldots,\eta_1 = 1$ for which the spacings between the 1s in $1\eta_n\eta_{n-1}\cdots\eta_1$ produce the ordered cycle sizes of  a $\theta$-biased derangement of length $n$. For the uniform case, the method described in~\cite{Mikawa2017} may also be described as a Markov chain (although it was not in that paper), and its transition matrix reduces to that in (\ref{prdef}) when $\theta = 1$; it is interesting to note that its construction differs dramatically from ours. The general method is compared to two other rejection-based methods, and shown to behave well.

We have focused here on the behavior of counts of cycle lengths, but the sequence $\eta_n,\eta_{n-1},\ldots,\eta_1 = 1$ may be used to generate the ordered permutation itself by a simple auxiliary randomisation~\cite[Chapter 5]{arratia_logarithmic_2003}: The first cycle starts with the integer 1, integers being chosen uniformly from the available unused integers at each $\eta_i = 0$, and closing the growing cycle and starting a new cycle with the smallest available integer at each $\eta_i = 1$. Probabilities associated with particular ordered cycle sizes may be computed from the Markov chain, and asymptotic properties of these lengths also follow directly.

Our chain does not generate derangements of size $n+1$ from one of size $n$, a property satisfied by the Feller Coupling (see the discussion after (\ref{feller})). Rather, the chain produces a derangement for a given value of $n$, and needs to be re-run to generate one of size $n+1$. 

Finally, we note that the Markov chain approach may be adapted to deal with other restricted patterns in the cycle lengths, such as requiring all cycles to have length at least $l$, or no cycles of length two or four.  


\bibliography{amsrefs1}
\bibliographystyle{cpclike}

\newpage
\appendix

\section{Monotonicity proofs}

This appendix provides analytical proofs for some of what has been observed in the simulation studies described in Section~\ref{examples}. More precisely, we prove $\alpha_n < \beta_n$ for $n=2b$, and show that the probability of having $A_1(n)\geq A_2(n)\geq \ldots \geq A_{\tilde K_n}(n)$ is strictly greater than the probability of $A_1(n)\leq A_2(n) \leq \ldots \leq A_{\tilde K_n}(n)$ for $n\geq 5$. 

For $i=4,\ldots,n-1$, the shift operator, denoted by $S_i^{(n)}:\Delta_n \longrightarrow \Delta_n$, is defined as follows. For any $r=(r_n,r_{n-1},\ldots,r_1)\in \Delta_n$ with $r_i=1$ and $r_{i-2}=0$, $$S_i^{(n)}(r)=S_i(r)=(r_n,\ldots,r_{i+1},0,1,0,r_{i-3},\ldots,r_1).$$ Let $S_i^{(n)}=S_i(r)=r$ for any other $r\in \Delta_n$. In other words, $S_i$ shifts $1$ at position $i$ of $r$ to position $i-1$, provided to not have a $1$ at position $i-2$. Let $Y^{(n)}=(Y^{(n)}_{n+1}=1,Y^{(n)}_{n},Y^{(n)}_{n-1},\ldots,Y^{(n)}_{1})$ be a $\{0,1\}-$valued Markov chain with transition probability matrix 
$$
\left(
\begin{array}{cc}
\bbbp(Y^{(n)}_{i} = 0 \mid Y^{(n)}_{i+1} = 0) & \bbbp(Y^{(n)}_{i} = 1 \mid Y^{(n)}_{i+1} = 0) \\
\bbbp(Y^{(n)}_{i}= 0 \mid Y^{(n)}_{i+1}= 1) & \bbbp(Y^{(n)}_{i} = 1 \mid Y^{(n)}_{i+1} = 1)
\end{array}
\right)
= \left(
\begin{array}{cc}
\displaystyle p_i&  \displaystyle q_i\\
1 & 0
\end{array}
\right),\\\\
$$ 
where $q_i=1-p_i$, for $i=3,..,n$, and $Y^{(n)}_{2}=0, Y^{(n)}_{1}=1$.

\begin{lemma}\label{LemmaA}
For $4\leq i\leq n-1$, let $r=(r_n,\ldots,r_1)\in \Delta_n$ with $r_i=1$ and $r_{i-2}=0$. Then $\bbbp(Y=S_i(r))>\bbbp(Y=r)$ if and only if $p_iq_{i-1}/p_{i-2}q_{i} > 1$.
\end{lemma}

From now on, we suppose for the chain $Y$, we have $p_iq_{i-1}/p_{i-2}q_{i} >1 $ for any $4\leq i \leq n-1$. Then it is not hard to see that for $n\geq 2k$, conditioned on having exactly $k$ cycles (i.e. $k+1$ $1$'s in $Y^{(n)}$), the most likely outcome of $Y^{(n)}$ is $10^{n-2k+1}1(01)^{k-1}$ and the one with the minimum chance of happening is $1(01)^{k-1}0^{n-2k+1}1$.

For $r=(r_n,r_{n-1},\ldots,r_1)\in\Delta_n$, denote by $|r|$ the number of $1$s in $r$ or the number of indices $i$ for which $r_i=1$, and let $\sigma_1(r)$ be the biggest $i<n+1$ for which $r_i=1$. Then by induction, for $j=2,\ldots,|r|$, let $\sigma_j(r)$ be the biggest index $i<\sigma_{j-1}(r)$ such that $r_i=1$. For convenience let $\sigma_0(r)=n+1$ and note that $\sigma_{|r|}(r)=1$ always.

\begin{lemma} \label{LemmaB}
Let $4\leq i\leq n-1$ and $r=(r_n,r_{n-1},\ldots,r_1)\in\Delta_n$ such that $r_i=1$ and $r_{i-2}=0$. Then
$$\bbbp(\eta=S_i(r))>\bbbp(\eta=r).$$
\end{lemma}

\begin{proof}

\begin{equation}\label{pshift}
\begin{split}
\frac{\p(\eta=S_i(r))}{\p(\eta=r)}&=\frac{\p(\eta_i=0 \mid \eta_{i+1}=0)\p(\eta_{i-1}=1\mid \eta_i=0)}{\p(\eta_{i-2}=0 \mid \eta_{i-1}=0)\p(\eta_{i}=1\mid \eta_{i+1}=0)}\\
&= \frac{\frac{1}{\theta+i-2}.\frac{\lambda_i(\theta)}{\lambda_{i-1}(\theta)+\frac{\theta}{\theta+i-2}\lambda_{i-2}(\theta)}}{\frac{1}{\theta+i-1}.\frac{\lambda_{i-1}(\theta)}{\lambda_{i-2}(\theta)+\frac{\theta}{\theta+i-3}\lambda_{i-3}(\theta)}}\\
&=\frac{\frac{1}{\theta+i-2}.\frac{i-1}{\theta+i-1}}{\frac{1}{\theta+i-1}.\frac{i-2}{\theta+i-2}}\\
&=\frac{i-1}{i-2}>1.
\end{split}
\end{equation}
since $\lambda_i(\theta)=\frac{i-1}{\theta+i-1}\left(\lambda_{i-1}(\theta)+\frac{\theta}{\theta+i-2}\lambda_{i-2}(\theta)\right)$, for $i\geq 3$.
\end{proof}

\begin{remark}
Note that we could easily use the conditional relation between $\eta$ and $\xi$ to prove Lemma~\ref{LemmaB}. Namely, for $r\in \Delta_n$ with $r_i=1$, $r_{i-2}=0$, the left hand side of Equality (\ref{pshift}) is equal to
\[
\frac{\p(\xi=S_i(r))}{\p(\xi=r)}=\frac{\p(\xi_i=0)\p(\xi_{i-1}=1)}{\p(\xi_i=1)\p(\xi_{i-1}=0)}=\frac{i-1}{i-2}.
\]
\end{remark}

Using multiple shifting or directly from the conditional relation of $\xi$ and $\eta$, it is trivial to conclude the following proposition.
\begin{proposition}\label{proportionA}
Let $r,r'\in \Delta_n$ with $|r|=|r'|=b$. Then

\[\frac{\p(\eta=r)}{\p(\eta=r')}=\frac{\p(\xi=r)}{\p(\xi=r')}=\prod\limits_{j=1}^{b-1}\frac{\sigma_j(r')-1}{\sigma_j(r)-1}
\]
\end{proposition}

Note that the proportion in the statement of Proposition \ref{proportionA} does not depend on $\theta$.

We say a sequence $r=(r_n,\ldots,r_1)\in \Delta_n$ is even if for any $i=1,\ldots,|r|$, $\sigma_{i-1}(r)-\sigma_i(r)$ is even. Similarly, $r$ is odd if for any $i=1,\ldots,|r|$, $\sigma_{i-1}(r)-\sigma_i(r)$ is odd. In other words, the length of all cycles in an even (odd, respectively) sequence is even (odd, respectively). We denote by $\mathcal{E}_n$ and $\mathcal{O}_n$ the set of all even elements and odd elements of $\Delta_n$, respectively. Also, let 
\[\mathcal{E}'_n=\{r\in \mathcal{E}_n: \exists i=0,\ldots,\lfloor (|r|-1)/2\rfloor \textrm{ s.t. } \sigma_{2i}(r)-\sigma_{2i+1}(r)=2
\}.\]
This means that for a sequence $r\in \mathcal{E}'_n$, the $2i+1$-st spacing between $1$s (i.e. $2i+1$-st cycle) in $r$, for at least one $i$ in $\{1,\ldots,\lfloor (|r|-1)/2\rfloor\}$ has length $2$. Let $n=2b$, and for $k\leq b$ denote 
\begin{eqnarray*}
\mathcal{E}_{n,k} & = & \{r\in\mathcal{E}_n: |r|=k\} \\
\mathcal{E}'_{n,k} & = & \{r\in\mathcal{E}'_n: |r|=k\} \\
\mathcal{O}_{n,k} & = & \{r\in\mathcal{O}_n: |r|=k\}
\end{eqnarray*}
For $l\leq \lfloor b/2\rfloor$, the shift mapping provides a one-one correspondence between $\mathcal{E}_{n,2l}\setminus \mathcal{E}'_{n,2l}$ and $\mathcal{O}_{n,2l}$. More precisely, by shifting $1$s at positions $\sigma_{2i+1}(r)$ of $r\in \mathcal{O}_{n,2l}$, for $i=0,1,\ldots,l-1$, to their right we get an element $r'\in\mathcal{E}_{n,2l}\setminus \mathcal{E}'_{n,2l}$, and by shifting $1$s at position $\sigma_{2i+1}(r')$ of $r'\in\mathcal{E}_{n,2l}\setminus \mathcal{E}'_{n,2l}$, for $i=0,1,2,\ldots,l-1$, to their left we get $r$. In other words, 
\[r'=S_{\sigma_{2l-1}(r)}\circ \cdots \circ S_{\sigma_{3}(r)}\circ S_{\sigma_{1}(r)}(r).\] For example $100100100100001 \in \mathcal{O}_{14,4}$ is mapped into $100010100010001 \in \mathcal{E}_{14,4}\setminus \mathcal{E}'_{14,4}$. Lemma \ref{LemmaA} implies $\p(Y^{(n)}=r')>\p(Y^{(n)}=r)$. On the other hand, $\mathcal{O}_{n,2l+1}$ is empty for $l\in\bbbz_+$. This leads us to $\p(Y^{(n)}\in \mathcal{E}_n)>\p(Y^{(n)}\in \mathcal{O}_n)$, for $n=2b$.

In particular, $\beta_n=\p(\eta\in\mathcal{E}_n)>\alpha_n=\p(\eta\in\mathcal{O}_n)$. In fact, we can say more about the relation of $\beta_n$ and $\alpha_n$. 

\begin{theorem}
Let $n=2b$. Then $\p(Y^{(n)}\in\mathcal{E}_n)>\p(Y^{(n)}\in\mathcal{O}_n)$. Furthermore, letting $u_n$ be the number of $1$s in $Y^{(n)}_n,\ldots,Y^{(n)}_1$, 
\begin{equation*}
\p(Y^{(n)}\in\mathcal{E}_n)- \p(Y^{(n)}\in\mathcal{E}'_n, u_n \ is \ even)-\p(Y^{(n)}\in\mathcal{E}_n, u_n \ is \ odd)\\
=\sum\limits_{r\in \mathcal{O}_n}\varphi(r) \p(Y^{(n)}=r),
\end{equation*} 
where $$\varphi(r)=\prod\limits_{i=0}^{\lfloor\frac{|r|}{2}\rfloor-1}\frac{p_{\sigma_{2i+1}(r)}q_{\sigma_{2i+1}(r)-1}}{q_{\sigma_{2i+1}(r)}p_{\sigma_{2i+1}(r)-2}}>1.$$ 
\end{theorem}

\begin{proof}
From the discussion above and Lemma~\ref{LemmaA} for $l\leq \lfloor b/2\rfloor$, $r\in\mathcal{O}_{n,2l}$, $r'=S_{\sigma_{2l-1}(r)}\circ \cdots\circ S_{\sigma_{3}(r)}\circ S_{\sigma_{1}(r)}(r)$,
\[\varphi(r)=\frac{\p(Y^{(n)}=r')}{\p(Y^{(n)}=r)}=\prod\limits_{i=0}^{\lfloor\frac{|r|}{2}\rfloor-1}\frac{p_{\sigma_{2i+1}(r)}q_{\sigma_{2i+1}(r)-1}}{q_{\sigma_{2i+1}(r)}p_{\sigma_{2i+1}(r)-2}}>1.\]
Therefore, \[\p(Y^{(n)}\in\mathcal{E}_{n,2l}\setminus \mathcal{E}_{n,2l}')=\sum\limits_{r\in\mathcal{O}_{n,2l}}\varphi(r)\p(Y^{(n)}=r).\]
Noting that $\mathcal{O}_{n,2l+1}$ is empty for $l\in\bbbz_+$, and summing two sides of the above equation over $l\leq \lfloor b/2 \rfloor$ finishes the proof.
\end{proof}

The following is an immediate consequence of the last theorem.
\begin{theorem}
Let $n=2b$. Then $\alpha_n<\beta_n$. Furthermore,
\begin{equation*}
\beta_n - \p(\eta\in\mathcal{E}'_n, \tilde{K}_n\ even)-\p(\eta\in\mathcal{E}_n, \tilde{K}_n \ odd)\\
=\sum\limits_{r\in \mathcal{O}_n}\left( \prod\limits_{i=0}^{\lfloor\frac{|r|}{2}\rfloor-1}\frac{\sigma_{2i+1}(r)-1}{\sigma_{2i+1}(r)-2} \right) \p(\eta=r).
\end{equation*}
 \end{theorem}

A sequence $r\in\Delta_n$ is weakly increasing (weakly decreasing, respectively) if for any $l=1,..,|r|-1$, $\sigma_{l-1}(r)-\sigma_{l}(r)\leq \sigma_{l}(r)-\sigma_{l+1}(r)$ ($\sigma_{l-1}(r)-\sigma_{l}(r)\geq \sigma_{l}(r)-\sigma_{l+1}(r)$, respectively). The set of all increasing (decreasing, respectively) sequences in $\Delta_n$ is denoted by $\Lambda_1(n)$ ($\Lambda_2(n)$, respectively). Note that $\Lambda_1(n)=\Lambda_2(n)=\Delta_n$ for $n=2,3,4$. The simulation results in Table~\ref{Tab9} suggest that $\p(\eta\in\Lambda_1(n))$ is significantly smaller than $\p(\eta\in\Lambda_2(n))$. To prove this, let us define the inversion operator $r\in\Delta \mapsto \overleftarrow{r}$ by reversing $r$, that is let $\overleftarrow{r}_i=r_{n+2-i}$, for $i=1,\ldots,n+1$. It is clear that the inversion operator induces a bijection on $\Delta_n$, and $r\in\Delta_n$ is weakly increasing if only if $\overleftarrow{r}$ is weakly decreasing,  that means the inversion operator also induces a bijection from $\Lambda_1(n)$ to $\Lambda_2(n)$. From Lemmas~\ref{LemmaA} and \ref{LemmaB}, for $n\geq 5$ and $r\in\Lambda_1(n)$,
\[
\p(Y^{(n)}=\overleftarrow{r})> \p(Y^{(n)}=r) 
\] 
and in particular,
\[
\p(\eta=\overleftarrow{r})> \p(\eta=r). 
\] 
Moreover, Proposition~\ref{proportionA} implies
\[
\frac{\p(\eta=\overleftarrow{r})}{\p(\eta=r)} = \prod_{i=1}^{|r|-1}\frac{\sigma_i(r)-1}{\sigma_i(\overleftarrow{r})-1} =
\prod_{i=1}^{|r|-1}\frac{\sigma_i(r)-1}{n+1-\sigma_i(r)}.
\]
\begin{theorem}
For $n\geq 5$, we have
\[
\p(Y^{(n)}\in\Lambda_2(n))>\p(Y^{(n)}\in\Lambda_1(n)).
\] 
In particular,
\[
\p(\eta\in\Lambda_2(n))>\p(\eta\in\Lambda_1(n))
\] 
and
\[
\p(\eta\in\Lambda_2(n))=\sum\limits_{r\in\Lambda_1(n)}\left( \prod\limits_{i=1}^{|r|-1}\frac{\sigma_i(r)-1}{n+1-\sigma_i(r)}\right) \p(\eta=r).
\]
\end{theorem}

\end{document}